\documentclass[letterpaper, 10 pt, conference]{ieeeconf} 	

\IEEEoverridecommandlockouts    

\newtheorem{theorem}{Theorem}[section]
\newtheorem{lemma}[theorem]{Lemma}
\newtheorem{definition}[theorem]{Definition}

\newtheorem{problem}{Problem}
\newtheorem{remark}[theorem]{Remark}
\newtheorem{assumption}{Assumption}

\newcommand{\mc}{\mathcal}

\newcommand{\rangesp}{\mathscr{R}}
\newcommand{\nullsp}{\mathscr{N}}

\newcommand{\rank}{\operatorname{rank}}
\newcommand{\real}{\mathbb{R}} 
\newcommand{\integ}{\mathbb{Z}}

\newcommand{\integnneg}{\mathbb{Z}_{\geq 0}}

\newcommand{\tsp}{\mathsf{T}} 
\newcommand{\pinv}{\dagger} 
\newcommand{\inv}{{\negat 1}} 
\newcommand{\negat}{\scalebox{0.75}[.9]{\( - \)}}

\newcommand*{\QEDB}{\hfill\ensuremath{\square}}
\newcommand*{\QEDBL}{\hfill\ensuremath{\blacksquare}}
\newcommand\oprocendsymbol{\hbox{$\square$}}
\newcommand\oprocend{\relax\ifmmode\else\unskip\hfill%
\fi\oprocendsymbol}
\newcommand{\map}[3]{#1: #2 \rightarrow #3}

\newcommand{\until}[1]{\{1,\dots,#1\}}

\newcommand{\norm}[1]{\Vert #1 \Vert}

\usepackage{hyperref}
\usepackage{graphicx,color}
\graphicspath{{./IMG/}{images/}{./IMG/SAGEMATH_PLOTS/}}
\usepackage{amsmath}
\usepackage{amssymb}
\usepackage{mathrsfs}
\usepackage{subfigure}
\usepackage{url}
\usepackage{booktabs}
\usepackage{array}
\usepackage[table]{xcolor}
\usepackage{mathtools} 		
\usepackage{epstopdf}
\usepackage{cite}
\usepackage[vlined,ruled]{algorithm2e}
\usepackage{upgreek}
\usepackage{dsfont}

\usepackage{epstopdf}
\title{\LARGE \textbf{Data-Driven Exact Pole Placement for Linear Systems}}

\author{Gianluca Bianchin\thanks{
The author is with the ICTEAM and the Department of Mathematical Engineering at the Universit\'e catholique de Louvain.}}

\begin{document}
\maketitle

\begin{abstract}
The exact pole placement problem concerns computing a feedback gain
that will assign 
%
%
the poles of a system, controlled via static state feedback, at 
a set of pre-specified locations.
This is a classic problem in feedback control and numerous 
methodologies have been proposed in the literature for cases 
where a model of the system to control is available. 
In this paper, we study the problem of computing feedback gains for 
pole placement (and, more generally, eigenstructure assignment) directly 
from experimental data. 
Interestingly, we show that the closed-loop poles can be placed 
\textit{exactly} at arbitrary locations without relying on any model 
description but by using only finite-length trajectories generated by 
the open-loop system. In turn, these findings imply that classical 
control objectives, such as feedback stabilization or meeting transient 
performance specifications, can be achieved without first identifying a 
system model. 
Numerical experiments demonstrate the benefits of the data-driven 
pole-placement approach as compared to its model-based counterpart.
\end{abstract}

\section{Introduction}
\label{sec:1}
Data-driven control methods enable the synthesis of feedback 
controllers directly from historical data generated by a physical 
systems, and thus elude the need to construct or identify a model for 
the underlying system to control. 
Data-driven approaches are especially useful in scenarios where 
first-principle models are difficult to derive or the identification 
task may lead to numerically-unreliable model 
parametrizations~\cite{VK-FP:21,FD-JC-IM:22}.
In these cases, data-driven methods set out a huge potential since 
controllers can be synthesized directly from data, and thus possible 
uncertainties in the identified model parameters shall not propagate 
when these parameters are used for control design.

Data-driven control synthesis is, by now, a well-investigated 
area of research (see, e.g., the representative 
works~\cite{CD-PT:19,GB-VK-FP:19,JC-JL-FD:19}).
Despite the availability of several techniques to synthesize various 
types of controllers from data, to the best of our 
knowledge, the problem of data-driven pole placement and the (more 
general) problem of data-driven eigenstructure assignment via static 
feedback have not been studied until now. 
The classical problem of pole placement consists in finding a static 
feedback gain matrix such that the poles of the closed-loop system are 
in a set of pre-specified locations; analogously, the problem of 
eigenstructure assignment is that of finding a static feedback gain 
such that the closed-loop system has a pre-specified set of eigenvalues 
and eigenvectors (hereafter named eigenstructure). 
Motivated by this background, in this paper we study the data-driven 
pole placement problem and the data-driven eigenstructure assignment 
problem. Our results show that it is possible to place the 
closed-loop eigenvalues \textit{exactly}  at arbitrary locations 
(in this context, ``exactly'' means that the closed-loop poles can be 
placed at exact locations, in contrast with cases where they can placed 
within certain regions) by using formulas that can be applied directly 
on data. Moreover, our results show that the data-driven eigenstructure 
assignment problem is feasible under the same conditions required for 
its model-based counterpart.

\textit{Paper contributions.}
This paper features two main contributions. First, we show that static 
feedback gains that place the poles at an arbitrary set of locations 
can be computed directly from data collected from finite-length 
open-loop control experiments. We remark that our formulas apply also to 
cases where the open-loop system is not stable. We provide an explicit formula to compute the feedback gain  
and we show that the problem is always feasible when the underlying 
system is controllable. 
Second, we study the eigenstructure assignment problem and we provide a 
necessary and sufficient condition to check when such a problem is 
feasible. Moreover, we provide an explicit formula to compute feedback 
gains that assign a pre-specified eigenstructure. 
Finally, as a minor contribution, we evaluate via numerical simulations  
the benefits of the proposed data-driven method as compared to 
model-based approaches.

\textit{Related work.}
Several techniques have been proposed to synthesize controllers 
from data while avoiding the need to identify the system model.
Solutions for static feedback control are studied 
in~\cite{TMM-PR:17,ST-SA-NR-MM:20}, 
the linear quadratic regulator (LQR) 
in~\cite{CD-PT:19}, model predictive control (MPC) 
in~\cite{JC-JL-FD:19,JB-JK-MAM-FA:20}, 
minimum-energy control laws in~\cite{GB-VK-FP:19}, 
trajectory tracking problems in~\cite{LX-MT-BG-GF:21}, 
distributed control problems in~\cite{AA-JC:20}, and 
feedback-optimization controllers are proposed 
in~\cite{GB-MV-JC-ED:21-tac}.
Some extensions to the case of nonlinear systems are presented 
in~\cite{JB-FA:20,MG-CD-PT:22}. 
Most of these methods exploit the ability to express future
trajectories of a linear system in terms of a 
sufficiently-rich past trajectory, as shown by the Fundamental 
Lemma~\cite{JCW-PR-IM-BDM:05}. With respect to this body of literature, 
in this work, we focus on the exact pole-placement problem.

The model-based exact pole placement problem has a long history; 
a non-exhaustive list of references 
includes~\cite{AP-RS-TN:15,JK-NN-PV:85,SB-ED:82,MR-SE-AB-FT:09}.
However, all these classical methods construct on a model-based 
description of the system to control; in contrast with these methods, 
our focus here is to derive formulas for pole placement that can be 
applied directly on data.
In line with this work is the recent contribution~\cite{SM-RH:21}, where 
the authors 
study the problem of placing the closed-loop poles in linear matrix 
inequality (LMI) regions; in contrast, in this work, we focus on placing 
the poles at \textit{exact} locations and, moreover, we address the 
eigenstructure assignment problem.


\section{Preliminaries}
\label{sec:2}

%
%
%
In this section, we recall some useful facts on behavioral system theory 
from~\cite{JCW-PR-IM-BDM:05}. 
Given a signal (time-series) $\map{z}{\integ}{\real^\sigma}$, and 
scalars $T \in \integnneg \cup \{+\infty \}$, 
$i \in \integnneg, i \leq T$, we denote the restriction of $z$ to 
the interval $[i, i+T-1 ]$ by 
$z_{[i,i+T-1]} := \{z(i),\dots, z({i+T-1})\}$ 
(notice that $z_{[i,i+T-1]}$ is a $T$-long signal). With a slight 
abuse of notation, we will also denote by 
$z_{[i,i+T-1]} := (z(i), \dots, z({i+T})) \in \real^{\sigma T}$ the 
vectorization of the signal $z_{[i,i+T-1]}$, where the distinction 
will be clear from the context. 
%
%
Given the $T-$long signal $z_{[i,i+T-1]}$, we denote the associated 
Hankel matrix with $L$ (block) rows by:
\scalebox{.86}{\parbox{.5\linewidth}{
\begin{align*}
\mc H_{L}(z_{[i,i+T-1]}) = \begin{bmatrix}
z(i) & z(i+1) & \hdots & z(i+T-L)\\
z(i+1) & z(i+2) & \hdots & z(i+T-L+1)\\
\vdots & \vdots & \ddots & \vdots\\
z(i+L\negat1) & z(i+L) & \hdots & z(i+T-1)
\end{bmatrix},
\end{align*}
}}
Notice that 
$\mc H_{L}(z_{[i,i+T]}) \in \real^{L\sigma \times (T-L+1)}$.
The following definition is instrumental for our analysis.

\begin{definition}{\bf \textit{(Persistently Exciting 
Signal~\cite{JCW-PR-IM-BDM:05})}}
The signal $z_{[i,i+T-1]} \in \real^\sigma$ is persistently exciting of 
order $L$ if the matrix $\mc H_L(z_{[i,i+T-1]})$ has full row rank 
$\sigma L$.\QEDB
\end{definition}
\smallskip
We note that persistence of excitation implicitly requires that 
the number of columns of $\mc H_L(z_{[i,i+T-1]})$ is non-smaller 
than the number of rows, or  $T-L+1 \geq L\sigma$.

We recall the following properties of linear dynamical systems subject 
to a persistently exciting input. 

\begin{lemma}{\textit{\textbf{(Fundamental Lemma~\cite[Thm 1]{JCW-PR-IM-BDM:05}})}}
\label{lem:fundLemmarankHankelMatrix}
Assume that the linear system $x(t+1)=Ax(t)+Bu(t)$ is controllable and  
let $(u_{[0,T-1]}, x_{[0,T-1]})$ be an input-state trajectory generated 
by this system. 
If $u_{[0,T-1]}$ is persistently exciting of order $n+d$, then:
\begin{align*}
\rank \begin{bmatrix} \mc H_1(x_{[0,T-1]})\\  \mc H_d(u_{[0,T-1]}) \end{bmatrix} 
= n + d m.
\end{align*}
\hfill $\Box$
\end{lemma}
\smallskip
This condition will play a fundamental role in the sequel.

\section{Problem setting}
\label{sec:3}

In this section, we formulate the problem of interest and discuss 
existing (model-based) techniques for its solution. 

\subsection{Problem formulation}
Consider the discrete-time linear time-invariant system:
\begin{align}\label{eq:open-loop}
x(t+1) = Ax(t) + Bu(t),
\end{align}
where $A \in \real^{n \times n}$ and $B \in \real^{n \times m}$ denote, 
respectively, the system and input matrices, and 
$\map{x}{\integnneg}{\real^{n}}$ and $\map{u}{\integnneg}{\real^{m}}$ 
denote, respectively, the state and input signals. 
We assume that $B$ has full column rank.
The behavior of~\eqref{eq:open-loop} is governed by the \textit{poles} 
of the system, that is, by the eigenvalues of $A$. It is often 
desirable to modify the poles of the system to obtain certain 
properties, such as system stability or a desired transient 
performance. 
This can be achieved by using a state-feedback control law of the form 
$u(t) = -K x(t) + v(t),$ where $\map{v}{\integnneg}{\real^{m}}$ is a 
new free input and $K \in \real^{m \times n}$ is called 
\textit{feedback gain}, which should be chosen so that the controlled 
system
\begin{align}\label{eq:closed-loop}
x(t+1) = (A-BK)x(t) + v(t),
\end{align}
has the desired poles. 
In line with~\cite{AP-RS-TN:15,JK-NN-PV:85,SB-ED:82,MR-SE-AB-FT:09}, we 
make the following assumption.

\begin{assumption}[\bf \textit{Desired set of pole locations}]
\label{as:set_pole_locations}
The set of desired pole locations contains $n$ complex numbers 
$\mc L=\{\lambda_1, \dots , \lambda_n\}$ and is closed under complex 
conjugation. 
\QEDB\end{assumption}
\smallskip 

The data-driven state-feedback pole placement problem is then 
formulated precisely as follows.

\begin{problem}[\bf \textit{Pole placement}]\label{prob:pole_placement}
Given a set of complex numbers $\mc L$ satisfying
Assumption~\ref{as:set_pole_locations} and historical data 
$\mc D = (u_{[0,T-1]},x_{[0,T-1]})$ generated 
by~\eqref{eq:open-loop}, find, when possible, a matrix 
$K \in \real^{m \times n}$ such that the eigenvalues of $A-BK$ are the 
elements of the set $\mc L$.
\QEDB\end{problem}
\smallskip

Conditions for the existence of solutions to the pole placement problem 
are well known~\cite{CTC:84}: a solution exists if and only if $\mc L$ 
contains all uncontrollable modes~\cite{CTC:84} of $(A,B).$
Thus, we will make the following assumption. 
\begin{assumption}[\bf \textit{Controllability}]
\label{as:controllability}
All modes of $(A,B)$ are controllable. 
\QEDB\end{assumption}
\smallskip

In the single-input case ($m = 1$), the solution to 
Problem~\ref{prob:pole_placement}, when it exists, is  
unique~\cite{CTC:84}. 
In the multi-input case $1<m<n$, the feedback gain $K$ that solves the 
pole placement problem is in general non-unique. 
One common way to select a particular $K$ within the ambiguity set is 
to choose the one that assigns the closed-loop eigenstructure:
\begin{align}\label{eq:eigenvector_condition}
(A-BK)X = X \Lambda,
\end{align}
where $\Lambda$ is an $n \times n$ diagonal matrix with spectrum 
given by $\mc L$ and $X$ is a non-singular matrix of associated 
closed-loop eigenvectors, chosen according to some notion of 
optimality.
For instance, the authors in\cite[Sec.~2.5]{JK-NN-PV:85} show that 
choosing a matrix of eigenvectors $X$ that is well-conditioned leads to 
pole locations that are robust against perturbations of the entries of 
$A$.
Motivated by this, in this paper we consider the data-driven 
state-feedback eigenstructure assignment problem, formulated precisely 
as follows.

\begin{problem}[\bf \textit{Eigenstructure assignment}]
\label{prob:eigenstructure_assignment}
Given a set of complex numbers $\mc L$ satisfying
Assumption~\ref{as:set_pole_locations}, a matrix of linearly 
independent eigenvectors $X,$ and historical data 
$\mc D = (u_{[0,T-1]},x_{[0,T-1]})$ generated by~\eqref{eq:open-loop}, 
find, when possible, a matrix $K \in \real^{m \times n}$ such 
that~\eqref{eq:eigenvector_condition} holds.
\QEDB\end{problem}

%

\subsection{Existing model-based pole-placement methods}
Several formulas have been proposed in the literature to solve the pole 
placement and eigenstructure assignment problems. Next, we will 
summarize some of the most celebrated. In what follows, we denote by 
$M^\pinv$ the Moore-Penrose inverse of matrix $M.$

\begin{enumerate}

\item Approach in~\cite[Thm 3]{JK-NN-PV:85}. Let 
$
B = \begin{bmatrix}
U_0, & U_1
\end{bmatrix}
\begin{bmatrix}
Z \\ 0
\end{bmatrix}
$
with $[U_0,~U_1]$ orthogonal and $Z$ nonsingular. 
Then, the following choice 
satisfies~\eqref{eq:eigenvector_condition}:
\begin{align}\label{eq:K1}
K = Z^\inv U_0^\tsp (X\Lambda X^\inv - A).
\end{align}

\item Approach in~\cite[Main Theorem]{SB-ED:82}.  Assume 
that $\lambda(A) \cap \lambda(\Lambda)=\emptyset$ and let $G$ and $X$ 
satisfy $AX - X \Lambda +BG=0$. Then, the following choice 
satisfies~\eqref{eq:eigenvector_condition}:
\begin{align}\label{eq:K2}
K = G X^\inv.
\end{align}

\item Approach in~\cite[Thm 1]{MR-SE-AB-FT:09}. 
Let $X$ be an invertible matrix that satisfies 
$(I - B B^\pinv)  (X \Sigma - AX) = 0.$
Then, the following choice 
satisfies~\eqref{eq:eigenvector_condition}:
\begin{align}\label{eq:K3}
K = B + (X \Sigma X^\inv - A).
\end{align}
\end{enumerate}

It is evident from~\eqref{eq:K1}-\eqref{eq:K3} (see also
Remark~\ref{rem:sensitivity-eigenvalues}) that to obtain a 
numerically-reliable $K$ using these formulas, matrices $(A,B)$ must be 
known with high precision.   

\begin{remark}\label{rem:sensitivity-eigenvalues}
It is possible to quantify the sensitivity of the eigenvalues of $A-BK$ 
against perturbations of the entries of $A$ or $B$ as follows. 
Let $\lambda$ denote a simple eigenvalue of $M:=A-BK$ with left and 
right eigenvectors $x$ and $y$, respectively.
Wilkinson~\cite{JW:88} showed that if a perturbation $\Delta M$ is made 
to the entries of $M$, then there exists a simple eigenvalue 
$\hat \lambda$ of $M + \Delta M$ such that
\begin{align*}
\vert \hat \lambda - \lambda \vert 
\leq \operatorname{cond} (\lambda, M)
\norm{\Delta M} + O (\norm{\Delta M}^2),
\end{align*}
where $\operatorname{cond} (\lambda, M) = \frac{\norm{x}\norm{y}}{\vert y^* x \vert}$ denotes the condition number of $\lambda$. Notice that 
$\operatorname{cond} (\lambda, M) \geq 1$ and 
$\operatorname{cond} (\lambda, M) = 1$ if and only if $M$ is a normal 
matrix, that is $M^\tsp M = M M^\tsp$.
Thus, in a first-order sense, perturbations in the 
entries of $A$ or $B$ lead to shifts in the eigenvalues of 
$A-BK$ as amplified by the condition number of the matrix of 
eigenvectors $X$.
\QEDB\end{remark}
\smallskip 

Since matrices $(A,B),$ in practice, must be first identified from 
(possibly noisy) historical data before the 
formulas~\eqref{eq:K1}-\eqref{eq:K3} can be applied, a promising way to 
reduce the sensitivity of the closed-loop pole locations is to bypass 
the system identification process and to develop methods for determining 
$K$ directly from the collected data. 
Motivated by this, the focus of this paper is 
on deriving direct formulas for pole placement from data that do 
not require the identification of matrices $A$ and $B$.

\section{Data-Driven pole placement}
\label{sec:4}

In this section, we will focus on Problem~\ref{prob:pole_placement}.
We will assume the availability of historical data 
$\mc D = (u_{[0,T-1]},x_{[0,T-1]})$ collected over $T$ time intervals 
generated by~\eqref{eq:open-loop}. 
In what follows, we will denote by $\rangesp\{M\}$ the range space 
generated by the columns of $M$ and by $\nullsp\{M\}$ the null space of 
the columns of $M$.
It will be useful to consider the following 
representation of the data:
\begin{align*}
U_0 &:= \begin{bmatrix}
u(0) & u(2) & \dots & u(T-2)
\end{bmatrix}
 \in \real^{m \times T-1},\\
X_0 &:= \begin{bmatrix}
x(0) & x(1) & \dots & x(T-2)
\end{bmatrix}
 \in \real^{n \times T-1},\\
X_1 &:= \begin{bmatrix}
x(1) & x(2) & \dots & x(T-1)
\end{bmatrix}
 \in \real^{n \times T-1}.
\end{align*}

\begin{theorem}[\bf \textit{Data-driven pole placement}]
\label{thm:pole_placement}
Let 
Assumptions~\ref{as:set_pole_locations}--\ref{as:controllability} 
be satisfied, $\mc L=\{\lambda_1, \dots , \lambda_n\}$, and 
$u_{[0,T-1]}$ be persistently exciting of order $n+1$.
Then, there exists a matrix 
$M = [m_1, \dots , m_n] \in \real^{T-1 \times n}$, with $\rank(M) = n,$
that satisfies:
\begin{align}\label{eq:conditionM}
0 &= (X_1 - \lambda_i X_0) m_i, && \forall i \in \until n.
\end{align}
Moreover, for any $M$ that satisfies~\eqref{eq:conditionM}, the matrix
\begin{align}\label{eq:pole_placement_formula}
K &= - U_0 M (X_0 M )^\pinv,
\end{align}
satisfies 
$\det(A - BK - \lambda I)=0$ for all $\lambda \in \mc L.$
\QEDB\end{theorem}

\begin{proof}
To prove existence of $M$, notice that
\begin{align}\label{eq:m_existence_condition_pt1}
(X_1 - \lambda_i X_0) m_i = 
\begin{bmatrix} A-\lambda_iI, & B \end{bmatrix}
\begin{bmatrix} X_0 \\ U_0 \end{bmatrix}
m_i.
\end{align}
Since $(A,B)$ is controllable, $\rank [A-\lambda_iI,~B] = n$ and
thus $[A-\lambda_iI,~B] $ has a nontrivial ($m$-dimensional) right null 
space. Thus, it is sufficient to choose the columns of $M$ so that:
\begin{align}\label{eq:m_existence_condition}
\begin{bmatrix} X_0 \\ U_0 \end{bmatrix}
m_i \in \nullsp\{\begin{bmatrix} A-\lambda_iI, & B \end{bmatrix}\}.
\end{align}
Since $u_{[0,T-1]}$ is persistently exciting of order $n+1,$ 
Lemma~\ref{lem:fundLemmarankHankelMatrix} guarantees 
$\rank [X_0^\tsp, U_0^\tsp]^\tsp = n+m$, and thus $m_i$ can always be 
chosen so that~\eqref{eq:m_existence_condition} holds, thus proving 
existence of $M$.
To show that $\rank(M) =n,$ notice that 
\begin{align}\label{eq:null_space_dimension}
\dim \nullsp \{
\begin{bmatrix} X_0 \\ U_0 \end{bmatrix}\}
\geq mn,
\end{align}
and thus there always exist $n$ linearly independent vectors $m_i$ that 
satisfy~\eqref{eq:m_existence_condition}.

To prove the second part of the claim, notice that
\begin{align}\label{eq:aux_proof1}
0&= (X_1 - \lambda_i X_0) m_i \nonumber\\
&= (AX_0 + BU_0 - \lambda_i X_0) m_i \nonumber\\
&= (A- \lambda_iI) X_0 m_i + BU_0 m_i,
\end{align}
where the last identity follows from $X_1 = AX_0 + BU_0,$ which holds 
because $X_0, X_1, U_0$ are generated by~\eqref{eq:open-loop}.
Next, by using~\eqref{eq:pole_placement_formula} we have
$-U_0 m_i = K X_0 m_i$. 
In fact, since $\rank(M)=n,$ $\rank(X_0M) =n$ and thus $(X_0 M)^\pinv$
is a right inverse of $X_0M.$
By substituting this identity 
into~\eqref{eq:aux_proof1} we obtain:
\begin{align*}
(A- BK- \lambda_iI) X_0 m_i =0,
\end{align*}
which proves the claim. 
\end{proof}

The formula~\eqref{eq:pole_placement_formula} provides an direct way 
to determine feedback gains by performing algebraic operations on the 
data and without first identifying $(A,B)$. 
The condition~\eqref{eq:conditionM} specifies a set of linear equations 
in the unknown $M,$ and thus $M$ can be determined by using standard 
linear equation solvers. 
Finally, we refer to Section~\ref{sec:6} for a discussion on 
the numerical benefits of utilizing~\eqref{eq:pole_placement_formula} 
as compared to the standard model-based pole placement formulas.

%
%
%
%
%
%

\section{Data-driven eigenstructure assignment}
\label{sec:5}

In this section, we will tackle the eigenstructure assignment problem.
%
It is natural to begin by asking ask under what conditions a given 
nonsingular matrix $X$ can be assigned as eigenvectors. 
The following result addresses this question.

\begin{theorem}[\bf \textit{Feasibility of eigenstructure assignment}]
Let Assumptions~\ref{as:set_pole_locations}--\ref{as:controllability}
hold, $X\in \real^{n \times n}$ be a nonsingular matrix, and assume 
that the input $u_{[0,T-1]}$ is persistently exciting of order $n+1$.
There exists a solution $K$
to~\eqref{eq:eigenvector_condition} if and only if 
\begin{align}\label{eq:deltaA}
\Delta A := A - X \Lambda X^\inv 
\in
\rangesp\{ X_1 \begin{bmatrix}
X_0 \\ U_0
\end{bmatrix}^\pinv
\begin{bmatrix}
0 \\ I_m
\end{bmatrix}
\}. 
\end{align}
\QEDB\end{theorem}

\begin{proof}
Notice that~\eqref{eq:eigenvector_condition} holds if and only if 
$$-BK = X\Lambda X^\inv -A.$$
Since $K$ is a free matrix, this holds if and only if the columns of
$X\Lambda X^\inv -A \in \rangesp\{B\}.$
To characterize $\rangesp\{B\},$ let $z \in \real^m$ be arbitrary, and 
notice that $Bz$ can be expressed as:
\begin{align}\label{eq:Bz}
Bz  = 
\begin{bmatrix} A & B \end{bmatrix}
\begin{bmatrix} 0 \\ z \end{bmatrix}
= \begin{bmatrix} A & B \end{bmatrix}
\begin{bmatrix} X_0 \\ U_0 \end{bmatrix} 
g,
\end{align}
for some $g \in \real^{T-1}.$ Here, the last identity follows by 
noting that, because $u_{[0,T-1]}$ is persistently exciting of order 
$n+1$, Lemma~\ref{lem:fundLemmarankHankelMatrix} guarantees 
$\rank \begin{bmatrix} X_0 \\ U_0 \end{bmatrix}  = n+m,$ and thus there 
always exists $g$ such that~\eqref{eq:Bz} holds. 
Since $g$ is guaranteed to exist, any $g$ that satisfies~\eqref{eq:Bz} 
can be expressed as:
\begin{align}\label{eq:expression-g}
g = \begin{bmatrix} X_0 \\ U_0 \end{bmatrix}^\pinv \begin{bmatrix} 0 \\ z \end{bmatrix} + w,
\end{align}
where $w$ satisfies $X_0w=U_0w=0.$
Next, notice that~\eqref{eq:Bz} can be re-expressed as: 
\begin{align}\label{eq:Bz-pt2}
Bz = \begin{bmatrix} A & B \end{bmatrix}
\begin{bmatrix} X_0 \\ U_0 \end{bmatrix} 
g = X_1 g.
\end{align}
By combining~\eqref{eq:expression-g} with \eqref{eq:Bz-pt2}, we 
obtain:
\begin{align*}
Bz = X_1 g  
&= X_1 
\begin{bmatrix} X_0 \\ U_0 \end{bmatrix}^\pinv \begin{bmatrix} 0 \\ z \end{bmatrix} + X_1 w\\
&= 
X_1 
\begin{bmatrix} X_0 \\ U_0 \end{bmatrix}^\pinv \begin{bmatrix} 0 \\ z \end{bmatrix} +  
\begin{bmatrix} A & B \end{bmatrix}
\begin{bmatrix} X_0 \\ U_0 \end{bmatrix}  w\\
&= X_1 
\begin{bmatrix} X_0 \\ U_0 \end{bmatrix}^\pinv \begin{bmatrix} 0 \\ z \end{bmatrix},
\end{align*}
where the last inequality follows from the properties of $w$.
Hence, we conclude
\begin{align*}
\rangesp\{B\} = 
\rangesp\{X_1 \begin{bmatrix}
X_0 \\ U_0
\end{bmatrix}^\pinv
\begin{bmatrix}
0 \\ I_m
\end{bmatrix}
\},
\end{align*}
which proves the claim.
\end{proof}

The theorem provides a characterization of all perturbations $\Delta A$
of the open-loop system matrix $A$ that can be obtained via static 
feedback: these are all and only the matrices that belong to the 
following space:
$$
\rangesp\{X_1 \begin{bmatrix}
X_0 \\ U_0
\end{bmatrix}^\pinv
\begin{bmatrix}
0 \\ I_m
\end{bmatrix}
\}.
$$
When the open-loop system matrix $A$ is known, the theorem also 
provides a condition to determine whether the eigenstructure assignment 
problem admits a solution: the problem is feasible if and only if 
$A - X \Lambda X^\inv $ belongs to the range space of the matrix 
characterized in~\eqref{eq:deltaA}. 
%

Before stating our result, we present the following technical 
lemma, which is a direct consequence of~\cite[Cor 1]{JK-NN-PV:85}.

\begin{lemma}\label{lem:space_dimension}
Let Assumptions~\ref{as:set_pole_locations}--\ref{as:controllability} 
hold, and $X\in \real^{n \times n}$ be a nonsingular matrix. 
Moreover, let $x_j$ denote the $j$-th column of $X$, corresponding to 
the assigned eigenvalue $\lambda_j \in \mc L.$
Then, $x_j \in \mc S_j$, where 
\begin{align*}
\mc S_j = \nullsp \{U^\tsp (A - \lambda_j I)\},
\end{align*}
where the columns of $U$ form a basis for $\nullsp\{B\}.$
Moreover, the dimension of $\mc S_j$ is:
\begin{align*}
\dim(\mc S_j) = m.
\end{align*}
\QEDB\end{lemma}

We remark that this lemma is of model-based nature, and thus the 
provided characterization is of no use when $A$ and $U$ are unknown. 
Despite its nature, in what follows we will next use this lemma for 
technical purposes (to derive necessary conditions for the 
eigenstructure assignment problem to be feasible and in the proof of the 
subsequent result). 
Since the maximum number of independent eigenvectors that can be 
chosen for each assigned eigenvalue is equal to $\dim (\mc S_j)= m,$
it follows that the algebraic multiplicity of the eigenvalue 
$\lambda_j \in \mc L$ to be assigned must be less than or equal to
$m$. The lemma thus motivates the following assumption.


\begin{assumption}[\bf \textit{Set of pole locations}]
\label{as:set_pole_locations_extended}
The desired pole locations $\mc L=\{\lambda_1, \dots , \lambda_n\}$ 
and eigenvectors $X$ satisfy:
\begin{enumerate}
\item $\mc L$ is closed under complex conjugation,
\item $\mc L$ contains $\nu$ complex numbers with associated 
algebraic multiplicities $\{m_1, \dots , m_\nu\}$ satisfying 
$m_1 + \dots + m_\nu = n,$ $m_i\leq m$ for all $i \in \until \nu,$
\item pairs of complex conjugate poles with $\lambda_i = \lambda_j^*$ 
satisfy $m_i = m_j$,
\item $X$ is such that the desired $(A-BK)$ is non-defective (i.e., it 
admits $n$ linearly independent eigenvectors).
\QEDB\end{enumerate}
\end{assumption}
\smallskip

With this technical assumption, we now provide the following formula 
for eigenstructure assignment.

\begin{theorem}[\bf \textit{Data-driven eigenstructure assignment}]
\label{thm:eigenstructure_assignment}
Let 
Assumptions~\ref{as:set_pole_locations}--\ref{as:set_pole_locations_extended} 
be satisfied, $\mc L=\{\lambda_1, \dots , \lambda_n\}$, and 
$u_{[0,T-1]}$ be persistently exciting of order $n+1$.
Then, there exists a matrix 
$M = [m_1, \dots , m_n] \in \real^{T-1 \times n}$, with $\rank(M) = n,$
that satisfies:
\begin{align}\label{eq:conditionM_eigenstructure}
0 &= (X_1 - \lambda_i X_0) m_i, && \forall i \in \until n, \nonumber\\
X &= X_0 M,
\end{align}
Moreover, for any $M$ that satisfies~\eqref{eq:conditionM}, the matrix
\begin{align}\label{eq:eigenstructure_assignment_formula}
K &= - U_0 M (X_0 M )^\pinv
\end{align}
satisfies~\eqref{eq:eigenvector_condition}.
\QEDB\end{theorem}

\begin{proof}
We begin by proving the existence of $M$. By iterating the steps 
in \eqref{eq:m_existence_condition_pt1}--\eqref{eq:m_existence_condition} for the first condition 
in~\eqref{eq:conditionM_eigenstructure}, we conclude that that a matrix 
$M$ that satisfies~\eqref{eq:conditionM_eigenstructure} exists if and 
only if the following two conditions hold  simultaneously: 
\begin{align*} 
& \begin{bmatrix} X_0 \\ U_0 \end{bmatrix}
m_i \in \nullsp(\begin{bmatrix} A-\lambda_iI, & B \end{bmatrix}),
&& 
x_i = X_0 m_i,
\end{align*}
where $x_i$ denotes the $i$-th row of $X$. Since 
Lemma~\ref{lem:space_dimension} guarantees that $x_i \in \nullsp\{B\}$ 
and~\eqref{eq:null_space_dimension} holds, we conclude that there exists 
at least $n$ linearly independent vectors that 
satisfy~\eqref{eq:conditionM_eigenstructure}. 
To prove the second part of the claim, notice:
\begin{align*}
0&= (X_1 - \lambda_i X_0) m_i \nonumber
= (AX_0 + BU_0 - \lambda_i X_0) m_i \nonumber\\
&= (A- \lambda_iI) X_0 m_i + BU_0 m_i,
\end{align*}
where the last identity follows from $X_1 = AX_0 + BU_0,$ which holds 
because $X_0, X_1, U_0$ are generated by~\eqref{eq:open-loop}.
Next, by using~\eqref{eq:pole_placement_formula} we have
$-U_0 m_i = K X_0 m_i$. 
In fact, since $\rank(M)=n,$ $\rank(X_0M) =n$ and thus $(X_0 M)^\pinv$
is a right inverse of $X_0M.$
By substituting this identity 
into~\eqref{eq:aux_proof1}:
\begin{align*}
(A- BK- \lambda_iI) X_0 m_i =0,
\end{align*}
from which $\lambda_i$ is an eigenvalue of 
$A-BK$ with eigenvector $X_0m_i$. The conclusion follows using 
$X = X_0 M.$ 
\end{proof}

The formula~\eqref{eq:eigenstructure_assignment_formula} provides an 
explicit way to determine feedback gains that assign the desired 
eigenstructure by performing algebraic computations on the data.
Notice that, with respect the conditions required for pole 
placement~\eqref{eq:conditionM}, assigning the eigenstructure imposes 
$n^2$ additional constraints on matrix $M,$ described by $X = X_0 M$. 
We remark that, similarly to~\eqref{eq:conditionM}, 
condition~\eqref{eq:conditionM_eigenstructure} specifies a set of 
linear equations in the unknown $M,$ and thus $M$ can be determined by 
using standard linear equation solvers.

\section{Numerical analysis}
\label{sec:6}
In this section, we illustrate the methods via 
numerical simulations on two test problems. First, we apply the 
formulas to stabilize the dynamics of a chemical reactor and, second, 
we compare the accuracy of the data-driven 
formulas with respect to a model-based approach. 

Consider the following linear model describing a chemical 
reactor and obtained by 
discretizing~\cite[Example 1]{JK-NN-PV:85} with unitary sample time:
\begin{align}\label{eq:AB_chemical_reactor}
A &= \begin{bmatrix}
    6.9771  &  2.0379  &  5.0672  & -2.2212\\
   -0.6941 &  -0.0434 &  -0.4738  &  0.3425\\
    0.2048 &   0.9081  &  0.3159  &  0.6172\\
   -0.5082 &   0.7106 &  -0.2000  &  0.8531
\end{bmatrix},\nonumber\\
B^\tsp  &= 
\begin{bmatrix}
    4.8874 &   1.4777 &    5.0448 &   4.6020\\
   -6.5545 &   0.5230  & -1.1389 &  -0.1133
\end{bmatrix}.
\end{align}
This system is unstable and the open-loop eigenvalues are:
\begin{align*}
\operatorname{eig}(A) = \{   7.0162,
    1.0798,
    0.0002,
    0.0065 \},
\end{align*}
and thus state feedback is required to stabilize the system. 
Therefore, we move two of the unstable modes into the unitary circle, 
keeping the original stable modes. We thus assign the set:
$
\mc L = \{   0.5, 
    0.3, 
    0.0002,
    0.0065 \}.$
Historical data $(u_{[0,T-1]},x_{[0,T-1]})$  is generated by simulating 
the open-loop system for $T=10$ time steps by applying i.i.id. Gaussian 
noise as the input signal and starting from zero initial conditions. 
The feedback gain obtained as in~\eqref{eq:pole_placement_formula} 
using the built-in \texttt{fsolve} routine in \texttt{Matlab R2022a} to 
solve~\eqref{eq:conditionM} is:
\begin{align}\label{eq:K_chemical_reactor}
K &=
\begin{bmatrix}
   -0.1758 &  -1.3970  &  2.8668 &  -2.4679\\
   -0.4441 &   0.2711 &   4.9848 &  -4.9424
\end{bmatrix},
\end{align}
leading to the closed-loop eigenvalues:
\begin{align*}
\operatorname{eig}(A-BK) = 
\{    0.4999,
    0.3001,
    0.0002,
	0.0066\}.
\end{align*}
We interpret the error between the desired pole locations in $\mc L$ 
and the spectrum of $A-BK$ as a numerical error due to the poor 
conditioning of the regression problem~\eqref{eq:conditionM} resulting 
from the use of data generated by an unstable system (whose state 
diverges over time). 
For example, after $t=10$ time steps, we observed
$\norm{x(10)} = 1.254 \times 10^{5}$, which makes the regression matrix 
in~\eqref{eq:conditionM} numerically unreliable.
To further illustrate this fact, Fig.~\ref{fig:numerical_error_vary_T} 
compares the accuracy of the closed-loop eigenvalues 
when~\eqref{eq:pole_placement_formula} is applied to data generated 
by an unstable system (orange lines) and when, instead, it is applied 
to data generated by a stable system (green lines). 
The latter is obtained by first 
stabilizing~\eqref{eq:AB_chemical_reactor} with $u(t) = -K x(t) +v(t)$
using $K$ given as in~\eqref{eq:K_chemical_reactor} and, 
subsequently, by using the input $v(t) = -K_2 x(t)$ 
(see~\eqref{eq:closed-loop}) with $K_2$ obtained by 
applying~\eqref{eq:pole_placement_formula} to data generated by $A-BK$.
As illustrated by the figure, the accuracy of the resulting pole 
locations deteriorates for increasing values of $T$ (i.e., by using 
increasingly-long trajectories)
when~\eqref{eq:pole_placement_formula} is applied to data generated by 
an unstable system and, on the other hand, the poles accuracy remains 
high when~\eqref{eq:pole_placement_formula} is applied to data 
generated by a stable system. 


\begin{figure}[t]
\centering 
\includegraphics[width=.9\columnwidth]{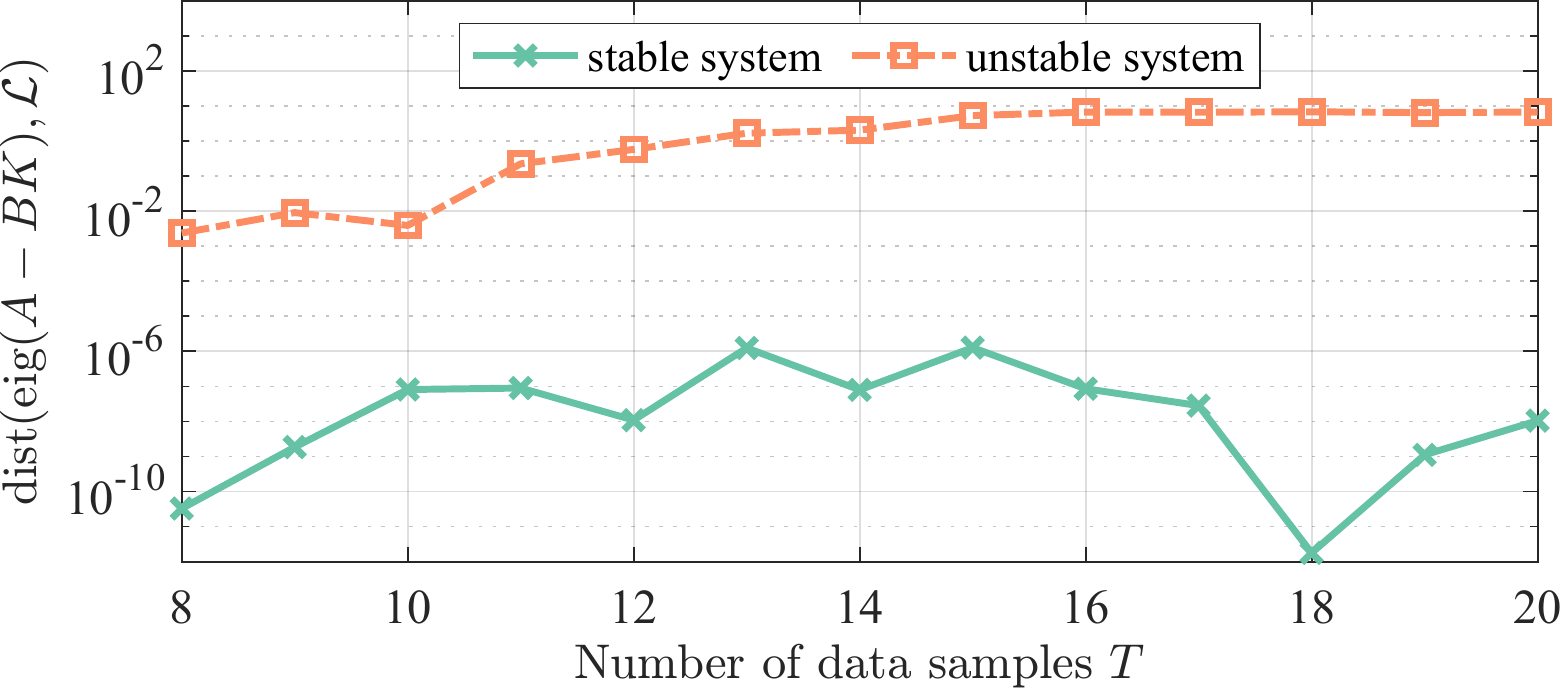}
\caption{Simulations comparing the accuracy of the 
closed-loop pole locations when~\eqref{eq:pole_placement_formula} is 
applied to open-loop data generated by a stable system (green line) and 
an unstable system (orange dashed line), for different values of the 
signal length $T$. Data for the unstable system has been generated by 
simulating the chemical reactor~\eqref{eq:AB_chemical_reactor}, while 
the data for the stable system has been generated by a pre-stabilized 
reactor model. The simulation illustrates 
that~\eqref{eq:pole_placement_formula} is more accurate by 
several orders of magnitude when applied to data generated by a stable 
system.}
\label{fig:numerical_error_vary_T}
\end{figure}

Next, we compare the accuracy of the closed-loop pole 
locations obtained using the data-based 
formula~\eqref{eq:pole_placement_formula} with those obtained 
using a model-based pole placement formula, applied to an identified 
model. In both cases, the methods are applied to noisy data, and we 
conducted Montecarlo simulations by averaging over $100$ experiments. 
To this aim, we generated noisy data by simulating: 
$$x(t+1) = A_0x(t)+B_0u(t)+e(t),$$ 
with $x(0) \sim \mc N(0,I_n)$, 
$u(t) \sim \mc N(0,I_m)$, $e(t)\sim \mc N(0,\sigma_e^2 I_n)$, 
and the matrices $A_0$ and $B_0$ have been chosen randomly and such 
that the modulus of all the eigenvalues of $A_0$ is inside the unit 
circle and $(A_0,B_0)$ is controllable, with $m=\lfloor n/2 \rfloor $. 
We identified $(A_0,B_0)$ from noisy data by solving the least-squares 
problem:
\begin{align*}
\begin{bmatrix} A & B \end{bmatrix}
\in \arg \min_{[A, B]}
\norm{X_1 - 
\begin{bmatrix} A & B \end{bmatrix}
\begin{bmatrix} X_0 \\ U_0\end{bmatrix} }_\text{F}.
\end{align*}
The set $\mc L$ has been chosen so that its entries are uniformly 
distributed in the real interval $[-n,n]$ and, for the model-based 
pole placement, we determined the feedback gain $K$ using the built-in 
\texttt{place} routine in \texttt{Matlab R2022a}.
Fig.~\ref{fig:model_based_vs_data_based} compares the accuracy of the 
closed-loop pole locations obtained by using a model-based placement 
formula and the data-based formula~\eqref{eq:pole_placement_formula} 
for increasing values of the state space size $n$.
As illustrated by the figure, the pole locations obtained using the 
data-based formula~\eqref{eq:pole_placement_formula} are more accurate 
by about one order of magnitude for all considered values of $n$. 
Moreover, by comparing the  results for  three different 
choices of the noise variance: $\sigma_e^2=1$ (top), $\sigma_e^2=10$ 
(middle), and $\sigma_e^2=100$ (bottom), the 
numerics suggest the higher the noise variance, the more the 
data-driven approach becomes preferable over the model-based one.


\begin{figure}[t]
\centering \subfigure[]{\includegraphics[width=1\columnwidth]{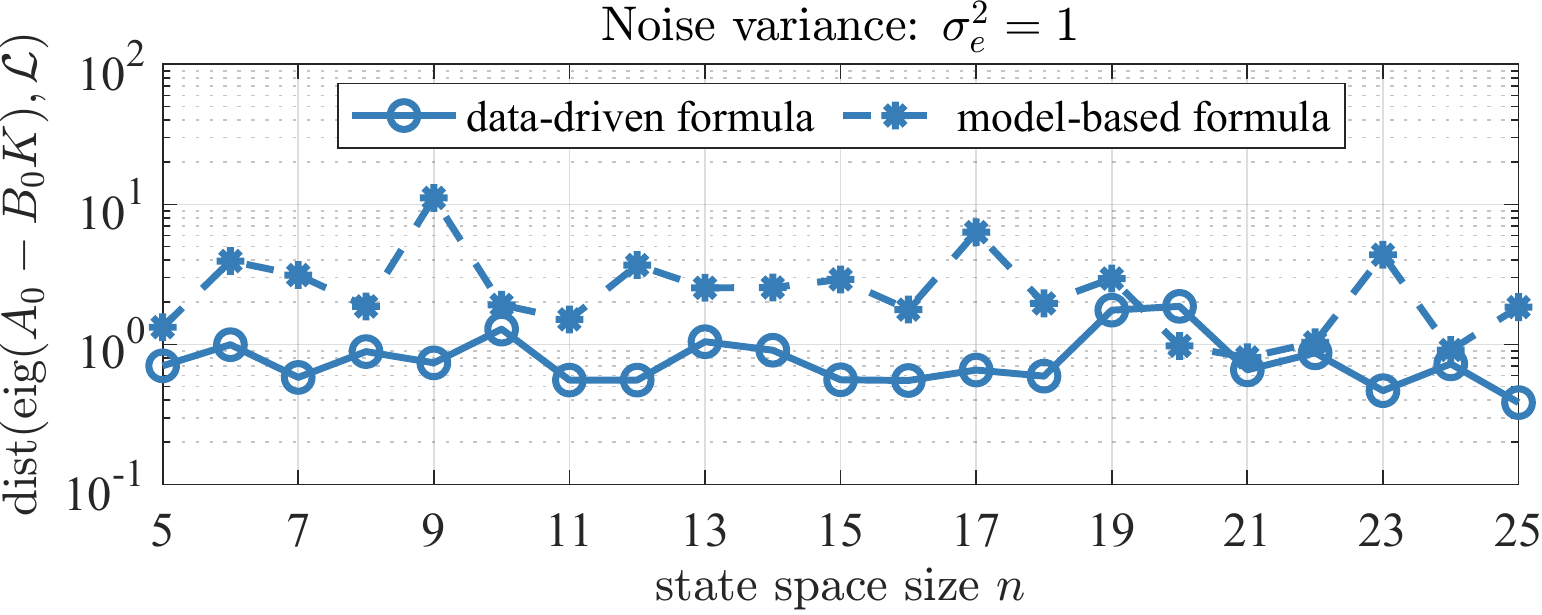}}\\
\centering \subfigure[]{\includegraphics[width=1\columnwidth]{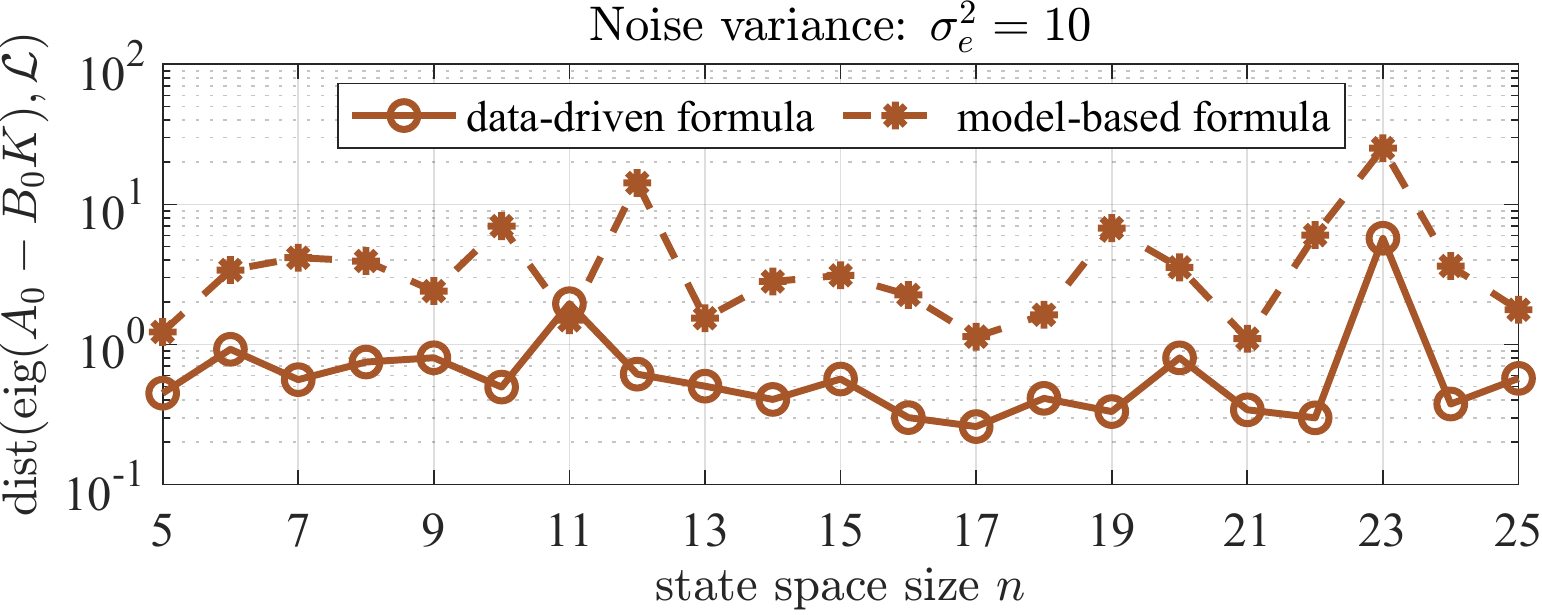}} \\
\centering \subfigure[]{\includegraphics[width=1\columnwidth]{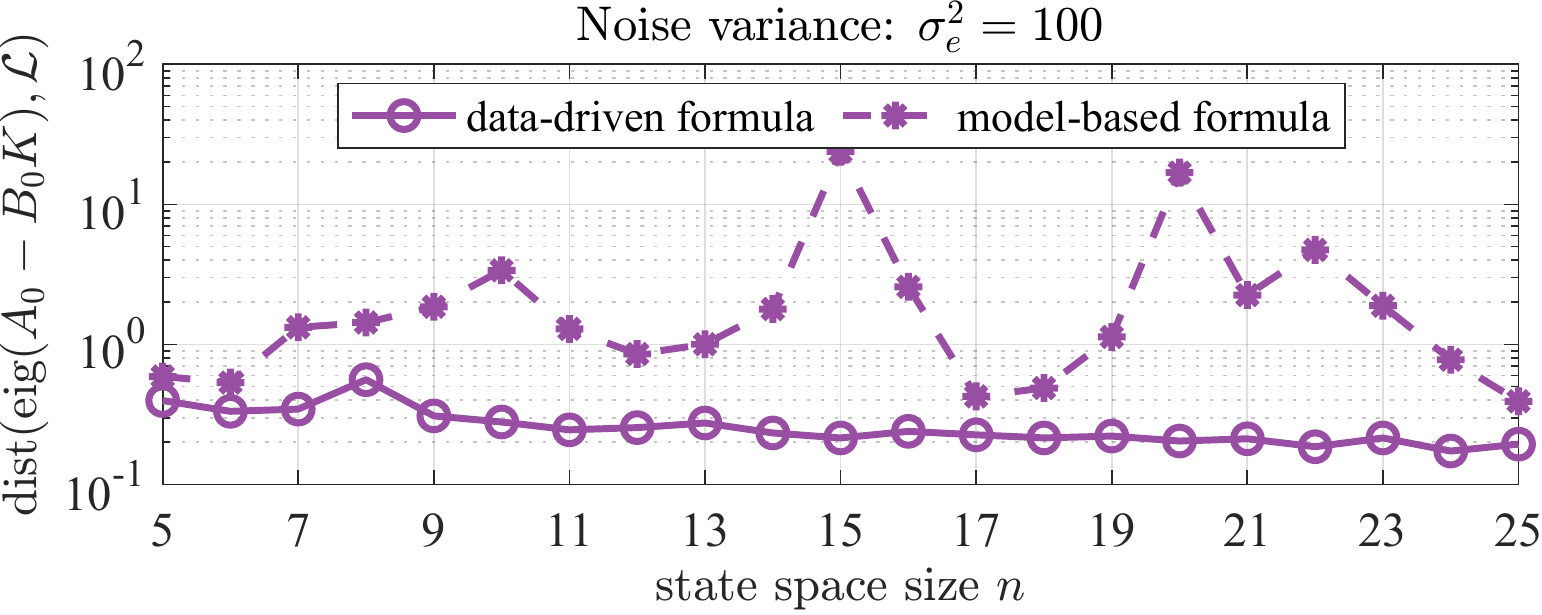}} 
\caption{Montecarlo simulation comparing the accuracy 
of~\eqref{eq:pole_placement_formula} when applied to noisy data with 
Gaussian distribution and three different levels of variance: 
$\sigma_e^2=1$ (top), $\sigma_e^2=10$ (middle), and $\sigma_e^2=100$ 
(bottom). The results suggest that the higher the noise variance, the 
more the data-driven formula becomes preferable over a model-based pole 
placement approach. }
\label{fig:model_based_vs_data_based}
\end{figure}

\section{Conclusions}
\label{sec:7}

In this paper, we derived data-driven formulas to compute static 
feedback gains matrices that assign arbitrarily the eigenstructure 
of a linear dynamical system. By leveraging the linearity of the 
dynamics and a persistence of excitation condition, we showed for the 
first time that the closed-loop eigenstructure can be assigned 
\textit{exactly}. Further, we illustrated the benefits of the 
data-driven methods, as compared to the model-based counterpart, 
through a set of numerical simulations, which showcase the numerical 
robustness of the approach, especially in the presence of noise in the 
measured data. This paper also opens several directions for future 
research, including an analytical investigation of the sensitivity of 
the closed-loop pole locations in the presence of noise in the data, 
and the extension to cases where the open-loop system contains 
uncontrollable modes.

\bibliographystyle{IEEEtran}
\bibliography{alias,main_GB,combined,GB,FP}

\begin{thebibliography}{10}
\providecommand{\url}[1]{#1}
\csname url@samestyle\endcsname
\providecommand{\newblock}{\relax}
\providecommand{\bibinfo}[2]{#2}
\providecommand{\BIBentrySTDinterwordspacing}{\spaceskip=0pt\relax}
\providecommand{\BIBentryALTinterwordstretchfactor}{4}
\providecommand{\BIBentryALTinterwordspacing}{\spaceskip=\fontdimen2\font plus
\BIBentryALTinterwordstretchfactor\fontdimen3\font minus
  \fontdimen4\font\relax}
\providecommand{\BIBforeignlanguage}[2]{{%
\expandafter\ifx\csname l@#1\endcsname\relax
\typeout{** WARNING: IEEEtran.bst: No hyphenation pattern has been}%
\typeout{** loaded for the language `#1'. Using the pattern for}%
\typeout{** the default language instead.}%
\else
\language=\csname l@#1\endcsname
\fi
#2}}
\providecommand{\BIBdecl}{\relax}
\BIBdecl

\bibitem{VK-FP:21}
V.~Krishnan and F.~Pasqualetti, ``On direct vs indirect data-driven predictive
  control,'' in \emph{{IEEE} Conf.\ on Decision and Control}, Austin, TX, Dec.
  2021, pp. 736--741.

\bibitem{FD-JC-IM:22}
F.~D{\"o}rfler, J.~Coulson, and I.~Markovsky, ``Bridging direct \& indirect
  data-driven control formulations via regularizations and relaxations,''
  \emph{IEEE Transactions on Automatic Control}, 2022, (Early access).

\bibitem{CD-PT:19}
C.~De~Persis and P.~Tesi, ``Formulas for data-driven control: Stabilization,
  optimality, and robustness,'' \emph{IEEE Transactions on Automatic Control},
  vol.~65, no.~3, pp. 909--924, 2019.

\bibitem{GB-VK-FP:19}
G.~Baggio, V.~Katewa, and F.~Pasqualetti, ``Data-driven minimum-energy controls
  for linear systems,'' \emph{IEEE Control Systems Letters}, vol.~3, no.~3, pp.
  589--594, 2019.

\bibitem{JC-JL-FD:19}
J.~Coulson, J.~Lygeros, and F.~D{\"o}rfler, ``Data-enabled predictive control:
  In the shallows of the {DeePC},'' in \emph{{E}uropean {C}ontrol
  {C}onference}, 2019, pp. 307--312.

\bibitem{TMM-PR:17}
T.~M. Maupong and P.~Rapisarda, ``Data-driven control: A behavioral approach,''
  \emph{Systems \& Control Letters}, vol. 101, pp. 37--43, 2017.

\bibitem{ST-SA-NR-MM:20}
S.~Talebi, S.~Alemzadeh, N.~Rahimi, and M.~Mesbahi, ``Online regulation of
  unstable {LTI} systems from a single trajectory,'' \emph{arXiv preprint},
  2020, arXiv:2006.00125.

\bibitem{JB-JK-MAM-FA:20}
J.~{Berberich}, J.~{Koehler}, M.~A. {M\"{u}ller}, and F.~{Allg\"{o}wer},
  ``Data-driven model predictive control with stability and robustness
  guarantees,'' \emph{IEEE Transactions on Automatic Control}, 2020, to appear.

\bibitem{LX-MT-BG-GF:21}
L.~Xu, M.~Turan~Sahin, B.~Guo, and G.~Ferrari-Trecate, ``A data-driven convex
  programming approach to worst-case robust tracking controller design,''
  \emph{arXiv preprint}, 2021, arXiv:2102.11918.

\bibitem{GB-MV-JC-ED:21-tac}
G.~Bianchin, M.~Vaquero, J.~Cort\'{e}s, and E.~Dall'Anese, ``Online stochastic
  optimization for unknown linear systems: Data-driven synthesis and controller
  analysis,'' \emph{arXiv preprint}, Aug. 2021, arXiv:2108.13040.

\bibitem{AA-JC:20}
A.~Allibhoy and J.~Cort{\'e}s, ``Data-based receding horizon control of linear
  network systems,'' \emph{IEEE Control Systems Letters}, vol.~5, no.~4, pp.
  1207--1212, 2020.

\bibitem{JB-FA:20}
J.~{Berberich} and F.~{Allg\"{o}wer}, ``A trajectory-based framework for
  data-driven system analysis and control,'' in \emph{{E}uropean {C}ontrol
  {C}onference}, 2020, pp. 1365--1370.

\bibitem{MG-CD-PT:22}
M.~Guo, C.~De~Persis, and P.~Tesi, ``Data-driven stabilization of nonlinear
  polynomial systems with noisy data,'' \emph{IEEE Transactions on Automatic
  Control}, 2022, (early access) arXiv:2011.07833.

\bibitem{JCW-PR-IM-BDM:05}
J.~C. Willems, P.~Rapisarda, I.~Markovsky, and B.~D. Moor, ``A note on
  persistency of excitation,'' \emph{Systems \& Control Letters}, vol.~54,
  no.~4, pp. 325--329, 2005.

\bibitem{AP-RS-TN:15}
A.~Pandey, R.~Schmid, and T.~Nguyen, ``Performance survey of minimum gain exact
  pole placement methods,'' in \emph{{E}uropean {C}ontrol {C}onference}, 2015,
  pp. 1808--1812.

\bibitem{JK-NN-PV:85}
J.~Kautsky, N.~K. Nichols, and P.~Van~Dooren, ``Robust pole assignment in
  linear state feedback,'' \emph{International Journal of control}, vol.~41,
  no.~5, pp. 1129--1155, 1985.

\bibitem{SB-ED:82}
S.~P. Bhattacharyya and E.~De~Souza, ``Pole assignment via {S}ylvester's
  equation,'' \emph{Systems \& Control Letters}, vol.~1, no.~4, pp. 261--263,
  1982.

\bibitem{MR-SE-AB-FT:09}
M.~A. Rami, S.~El~Faiz, A.~Benzaouia, and F.~Tadeo, ``Robust exact pole
  placement via an {LMI}-based algorithm,'' \emph{IEEE Transactions on
  Automatic Control}, vol.~54, no.~2, pp. 394--398, 2009.

\bibitem{SM-RH:21}
S.~Mukherjee and R.~R. Hossain, ``Data-driven pole placement in {LMI} regions
  with robustness guarantees,'' in \emph{{IEEE} Conf.\ on Decision and
  Control}, 2022, pp. 4010--4015.

\bibitem{CTC:84}
C.-T. Chen, \emph{Linear System: Theory and Design}.\hskip 1em plus 0.5em minus
  0.4em\relax Holt, Rinehart, and Winston, 1984.

\bibitem{JW:88}
J.~H. Wilkinson, \emph{The algebraic eigenvalue problem}.\hskip 1em plus 0.5em
  minus 0.4em\relax London: Oxford University Press, 1988.

\end{thebibliography}
\end{document}